\documentclass{amsart}

\usepackage{amsmath,amssymb,amsthm,accents}
\usepackage{amscd}
\usepackage[initials,short-journals,non-sorted-cites]{amsrefs}
\usepackage{enumitem}
\usepackage{type1cm}
\usepackage{float}
\usepackage[all]{xy}
\usepackage{amsfonts}
\usepackage[dvipdfmx]{graphicx}
\usepackage{comment}
\usepackage{bm}

\usepackage{color}

\allowdisplaybreaks[4]

\renewcommand{\abstract}[1]{
	\begin{center}
	\parbox{13cm}{\small {\sc Abstract.} #1}
	\end{center}
	\smallskip
}

\newtheorem{theo}{Theorem}[section]
\newtheorem{prop}[theo]{Proposition}
\newtheorem{lemm}[theo]{Lemma}
\newtheorem{cor}[theo]{Corollary}
\newtheorem{prob}[theo]{Problem}

\theoremstyle{definition}

\newtheorem{exam}[theo]{Example}

\newtheorem{remark}[theo]{Remark}

\newcommand{\Z}{\mathbb{Z}}

\newcommand{\R}{\mathbb{R}}

\title[Good involutions of generalized Alexander qundles]{Good involutions of generalized Alexander quandles}
\author{Yuta Taniguchi}
\address{\scriptsize DEPARTMENT OF MATHEMATICS, GRADUATE SCHOOL OF SCIENCE, OSAKA UNIVERSITY, 1-1, MACHIKANEYAMA, TOYONAKA, OSAKA, 560-0043, JAPAN}
\email{yuta.taniguchi.math@gmail.com}

\begin{document}

\maketitle

\abstract{Quandles with good involutions, which are called symmetric quandles, can be used to define invariants of unoriented knots and links. In this paper, we determine the necessary and sufficient condition for good involutions of a generalized Alexander quandle to exist. Moreover, we classify all good involutions of a connected generalized Alexander quandle up to symmetric quandle isomorphisms.}

\section{Introduction}
A {\it quandle} \cite{Joy,Mat} is a set $X$ with a binary operation $X\times X\to X;(x,y)\mapsto x^y$ satisfying the following three conditions.
\begin{itemize}
\item[(Q1)] For every $x\in X$, $x^x=x$.
\item[(Q2)] There exists a binary operation $X\times X\to X;(x,y)\mapsto x^{y^{-1}}$ such that $(x^y)^{y^{-1}}=(x^{y^{-1}})^y=x$ for every $x,y\in X$.
\item[(Q3)] For every $x,y,z\in X$, $(x^y)^z=(x^z)^{y^z}$.
\end{itemize}
These axioms were motivated from Reidemeister moves in knot theory. In this paper, we denote $(x^y)^z$ by $x^{yz}$. A quandle $X$ is a {\it kei} or an {\it involutory quandle} if for any $x,y\in X$, we have $x^y=x^{y^{-1}}$.

 An involution $\rho$ of a quandle $X$ is a {\it good involution} \cite{Kamada} if for every $x,y\in X$, it holds that $\rho(x^y)=\rho(x)^y$ and $x^{\rho(y)}=x^{y^{-1}}$. When $\rho$ is a good involution of a quandle $X$, we call the pair $(X,\rho)$ a {\it symmetric quandle}~\cite{Kamada-Oshiro}. A typical example of symmetric quandles is a pair $(X,{\rm Id}_X)$ of a kei $X$ and the identity map ${\rm Id}_X:X\to X$ (see \cite{Kamada,Kamada-Oshiro}).

A map $f:(X,\rho)\to(X^\prime,\rho^\prime)$ of symmetric quandles is called a {\it symmetric quandle isomorphism} if $f$ is a bijection, $f(x^y)=f(x)^{f(y)}$ for any $x,y\in X$ and $f\circ\rho=\rho^\prime\circ f$. When there is a symmetric quandle isomorphism between $(X,\rho)$ and $(X^\prime,\rho^\prime)$, we say that $(X,\rho)$ and $(X^\prime,\rho^\prime)$ are {\it symmetric quandle isomorphic}, denoted by $(X,\rho)\cong(X^\prime,\rho^\prime)$.

Given a quandle $X$, let $\mathcal{SQ}(X)$ be the set of all symmetric quandle isomorphism classes of $(X,\rho)$'s, that is,
\[
\mathcal{SQ}(X):=\{ (X,\rho)\mid \rho:\mbox{\rm a good involution of }X\}/\cong.
\]
Then, the following problems naturally arise:
\begin{prob}
\label{prob_1}
Determine the necessary and sufficient condition for good involutions of a quandle $X$ to exist.
\end{prob}
\begin{prob}
\label{prob_2}
Determine $\mathcal{SQ}(X)$ for a quandle $X$.
\end{prob}
In this paper, we give an answer of Problem~\ref{prob_1} when a quandle $X$ is the {\it generalized Alexander quandle} ${\rm GAlex}(G,\varphi)$ for some group $G$ and group automorphism $\varphi:G\to G$.
\begin{theo}
\label{main_theo:1}
Let $G$ be a group and $\varphi:G\to G$ a group automorphism. There exists a good involution of $\mathrm{GAlex}(G,\varphi)$ if and only if the quandle $\mathrm{GAlex}(G,\varphi)$ is a kei. 
\end{theo}
Moreover, if a generalized Alexander quandle $\mathrm{GAlex}(G,\varphi)$ is {\it connected}, we determine the set $\mathcal{SQ}(\mathrm{GAlex}(G,\varphi))$ (Theorem~\ref{main_theo:2}).

\section{Good involutions of the generalized Alexander quandle}
Let $G$ be a group and $\varphi:G\to G$ a group automorphism. The {\it generalized Alexander quandle} of $(G,\varphi)$, which is denoted by ${\rm GAlex}(G,\varphi)$, is the set $G$ with the binary operation $x^y:=\varphi(x)\varphi(y^{-1})y$. We note that the operation $x^{y^{-1}}$ is given by $x^{y^{-1}}=\varphi^{-1}(x)\varphi^{-1}(y)y$. If the group $G$ is an additive group, we also call ${\rm GAlex}(\varphi)$ the {\it Alexander quandle} of $(G,\varphi)$. At first, we will prove Theorem~\ref{main_theo:1}.
\begin{proof}[Proof of Theorem~\ref{main_theo:1}]
If the generalized Alexander quandle ${\rm GAlex}(G,\varphi)$ is a kei, then the identity map ${\rm Id}_G:G\to G$ is a good involution. Hence, we will show the only if part.

Let $\rho$ be a good involution of $\mathrm{GAlex}(G,\varphi)$. For any $x\in\mathrm{GAlex}(G,\varphi)$, we have $\rho(x)=\rho(x^ x)=\rho(x)^x=\varphi(\rho(x))\varphi(x^{-1})x$. Thus, it holds that $\varphi(\rho(x)^{-1})\rho(x)=\varphi(x^{-1})x$ for any $x\in\mathrm{GAlex}(G,\varphi)$.

Hence, we see that
\[
x^y=\varphi(x)\varphi(y^{-1})y=\varphi(x)\varphi(\rho(y)^{-1})\rho(y)=x^{\rho(y)}=x^{y^{-1}}
\]
for any $x,y\in\mathrm{GAlex}(G,\varphi)$, which impies that $\mathrm{GAlex}(G,\varphi)$ is a kei.
\end{proof} 
\begin{cor}
\label{cor:generalized_Alexander}
Let $G$ be a group and $\varphi$ a group automorphism of $G$. If ${\rm GAlex}(G,\varphi)$ is not a kei, the set $\mathcal{SQ}({\rm GAlex}(G,\varphi))$ is the empty set.
\end{cor}
\begin{proof}
This corollary follows immediately from Theorem~\ref{main_theo:1}.
\end{proof}
Next, we will discuss good involutions of the generalized Alexander quandle ${\rm GAlex}(G,\varphi)$ when ${\rm GAlex}(G,\varphi)$ is connected. A quandle $X$ is {\it connected} if for any $x,y\in X$, there are elements $y_1,\ldots,y_n\in X$ and $\varepsilon_1,\ldots,\varepsilon_n\in\{\pm 1\}$ such that $x^{y_1^{\varepsilon_1}\cdots y_n^{\varepsilon_n}}=y$.
\begin{prop}
\label{prop:good_inv}
Let $G$ be a group and $\varphi:G\to G$ a group automorphism. If ${\rm GAlex}(G,\varphi)$ is a kei and a connected quandle, there exists a bijection 
\[
\{\rho:\mbox{\rm good involutions of GAlex}(G,\varphi)\}\overset{1:1}{\leftrightarrow}\{r\in G\mid \varphi(r)=r,r^2=e\}.
\]
\end{prop}
\begin{proof}
Let $\rho:{\rm GAlex}(G,\varphi)\to{\rm GAlex}(G,\varphi)$ be a good involution and $e$ the identity element of $G$. We put $r:=\rho(e)$. Then, it holds that
\[
\varphi(r)=\varphi(\rho(e))=\varphi(\rho(e))\varphi(e^{-1})e=\rho(e)^e=\rho(e^e)=\rho(e)=r.
\]
 This implies that for any $x,y\in{\rm GAlex}(G,\varphi)$, we have
\[
(rx)^y=\varphi(rx)\varphi(y^{-1})y=r\varphi(x)\varphi(y^{-1})y=r(x^y)
\]
and
\[
(rx)^{y^{-1}}=\varphi^{-1}(rx)\varphi^{-1}(y^{-1})y=r\varphi^{-1}(x)\varphi^{-1}(y^{-1})y=r(x^{y^{-1}}).
\]
Since ${\rm GAlex}(G,\varphi)$ is connected, for each $x\in{\rm GAlex}(G,\varphi)$, there exist $x_1,x_2,\ldots,x_n\in{\rm GAlex}(G,\varphi)$ and $\varepsilon_1,\varepsilon_2\ldots,\varepsilon_n\in\{\pm 1\}$ such that $e^{x_1^{\varepsilon_1}x_2^{\varepsilon_2}\cdots x_n^{\varepsilon_n}}=x$. Then, we see that 
\begin{eqnarray*}
rx&=&re^{x_1^{\varepsilon_1}x_2^{\varepsilon_2}\cdots x_n^{\varepsilon_n}}\\
&=&(re)^{x_1^{\varepsilon_1}x_2^{\varepsilon_2}\cdots x_n^{\varepsilon_n}}\\
&=&r^{x_1^{\varepsilon_1}x_2^{\varepsilon_2}\cdots x_n^{\varepsilon_n}}\\
&=&\rho(e)^{x_1^{\varepsilon_1}x_2^{\varepsilon_2}\cdots x_n^{\varepsilon_n}}\\
&=&\rho(e^{x_1^{\varepsilon_1}x_2^{\varepsilon_2}\cdots x_n^{\varepsilon_n}})=\rho(x).
\end{eqnarray*}
 Thus, it holds that $\rho(x)=rx$. Since $\rho$ is an involution, we have $r^2=e$.

Let $r$ be an element of $G$ satisfying that $\varphi(r)=r$ and $r^2=e$. Then, we define a map $\rho_r:{\rm GAlex}(G,\varphi)\to {\rm GAlex}(G,\varphi)$ by $\rho_r(x)=rx$ for any $x\in {\rm GAlex}(G,\varphi)$. We will show that $\rho_r$ is a good involution. It is obvious that $\rho_r$ is an involution. Since $\varphi(r)=r$, for any $x,y\in {\rm GAlex}(G,\varphi)$, we have 
\[
\rho_r(x)^y=(rx)^y=\varphi(rx)\varphi(y^{-1})y=r\varphi(x)\varphi(y^{-1})y=r(x^y)=\rho_r(x^y)
\]
and
\[
x^{\rho_r(y)}=x^{ry}=\varphi(x)\varphi((ry)^{-1})ry=\varphi(x)\varphi(y^{-1}r^{-1})ry=\varphi(x)\varphi(y^{-1})y=x^y=x^{y^{-1}}.
\]
 Thus, the involution $\rho_r$ is a good involution. By the above discussion, we see that the map $r\mapsto\rho_r$ is bijective.
\end{proof}
\begin{remark}
In Proposition~\ref{prop:good_inv}, the assumption that a generalized Alexander quandle is connected is necessary.  For an additive group $G$, we denote the group automorphism $G\to G;g\mapsto -g$ by ${\rm inv}(G)$. We consider the quandle $R_4=(\Z/4\Z,{\rm inv}(\Z/4\Z))$, which is called the {\it dihedral quandle of order 4}. It is known that the quandle $R_4$ is not connected. By \cite{Kamada-Oshiro}, $R_4$ has four good involutions. However, the set $\{ r\in\Z/4\Z\mid {\rm inv}(\Z/4\Z)(r)=r,2r=0\}$ is equal to the set $\{0,2\}\subset\Z/4\Z$. 
\end{remark}
 Let $X$ be a quandle. A map $f:X\to X$ is a {\it quandle automorphism} if it is bijective and $f(x^y)=f(x)^{f(y)}$ for any $x,y\in X$. To determine the set $\mathcal{SQ}({\rm GAlex}(G,\varphi))$, we show the following lemma.
\begin{lemm}[cf. \cite{Higashitani-Kurihara}]
\label{lemm:quandle_auto}
Let $G$ be a group and $\varphi:G\to G$ a group automorphism. Suppose that ${\rm GAlex}(G,\varphi)$ is a connected quandle. Let $f:{\rm GAlex}(G,\varphi)\to{\rm GAlex}(G,\varphi)$ be a quandle automorphism. Then, the map $f_{\#}:{\rm GAlex}(G,\varphi)\to{\rm GAlex}(G,\varphi)$ defined by $f_{\#}(x)=f(x)f(e)^{-1}$ is a group automorphism of $G$.
\end{lemm}
\begin{proof}
Let $f:{\rm GAlex}(G,\varphi)\to{\rm GAlex}(G,\varphi)$ be a quandle automorphism. At first, we will show that the map $f_{\#}:{\rm GAlex}(G,\varphi)\to{\rm GAlex}(G,\varphi);x\mapsto f(x)f(e)^{-1}$ is also a quandle automorphism. For any $x,y\in {\rm GAlex}(G,\varphi)$, we have
\begin{eqnarray*}
f_{\#}(x)^{f_{\#}(y)}&=&(f(x)f(e)^{-1})^{f(y)f(e)^{-1}}\\
&=&\varphi(f(x)f(e)^{-1})\varphi((f(y)f(e)^{-1})^{-1})f(y)f(e)^{-1}\\
&=&\varphi(f(x)f(e)^{-1})\varphi(f(e)f(y)^{-1})f(y)f(e)^{-1}\\
&=&\varphi(f(x))\varphi(f(y)^{-1})f(y)f(e)^{-1}\\
&=&(f(x)^{f(y)})f(e)^{-1}=f(x^y)f(e)^{-1}=f_\#(x^y).
\end{eqnarray*}
Hence, the map $f_\#$ is a quandle automorphism. Notice that $f_\#(e)=e$ and $f_\#$ is a bijective. Since $\varphi(x)=x^e$ for any $x\in{\rm GAlex}(G,\varphi)$, it holds that 
\[
f_\#(\varphi(x))=f_{\#}(x^e)=f_\#(x)^{f_\#(e)}=f_\#(x)^{e}=\varphi(f_\#(x))
\]
for any $x\in {\rm GAlex}(G,\varphi)$.

Let $x$ and $y$ be elements of ${\rm GAlex}(G,\varphi)$. Since ${\rm GAlex}(G,\varphi)$ is connected, there exist $y_1,\ldots,y_n\in{\rm GAlex}(G,\varphi)$ and $\varepsilon_1,\ldots,\varepsilon_n\in\{\pm 1\}$ such that $y=e^{y_1^{\varepsilon_1}\cdots y_n^{\varepsilon_n}}$.  Using the relation $e^{e^{\pm 1}}=e$, we may assume that $\sum^n_{i=1}\varepsilon_i=0$. 

For each $z\in{\rm GAlex}(G,\varphi)$, we see that 
\[
\varphi(x)e^z=\varphi(x)\varphi(e)\varphi(z^{-1})z=\varphi(x)\varphi(z^{-1})z=x^z
\]
and
\[
\varphi^{-1}(x)e^{z^{-1}}=\varphi^{-1}(x)\varphi^{-1}(e)\varphi^{-1}(z^{-1})z=\varphi^{-1}(x)\varphi^{-1}(z^{-1})z=x^{z^{-1}}.
\]
Thus, we have
\begin{eqnarray*}
x^{y_1^{\varepsilon_1}\cdots y_{n-2}^{\varepsilon_{n-2}}y_{n-1}^{\varepsilon_{n-1}}y_n^{\varepsilon_n}}&=&\varphi^{\varepsilon_n}(x^{y_1^{\varepsilon_1}\cdots y_{n-2}^{\varepsilon_{n-2}}y_{n-1}^{\varepsilon_{n-1}}})e^{y_n^{\varepsilon_n}}\\
&=&\varphi^{\varepsilon_n}(\varphi^{\varepsilon_{n-1}}(x^{y_1^{\varepsilon_1}\cdots y_{n-2}^{\varepsilon_{n-2}}})e^{y_{n-1}^{\varepsilon_{n-1}}})e^{y_n^{\varepsilon_n}}\\
&=&\varphi^{\varepsilon_{n-1}+\varepsilon_n}(x^{y_1^{\varepsilon_1}\cdots y_{n-2}^{\varepsilon_{n-2}}})\varphi^{\varepsilon_n}(e^{y_{n-1}^{\varepsilon_{n-1}}})e^{y_n^{\varepsilon_n}}\\
&=&\varphi^{\varepsilon_{n-1}+\varepsilon_n}(x^{y_1^{\varepsilon_1}\cdots y_{n-2}^{\varepsilon_{n-2}}})e^{y_{n-1}^{\varepsilon_{n-1}}y_n^{\varepsilon_n}}\\
&=&\varphi^{\sum^n_{i=1}\varepsilon_i}(x)e^{y_1^{\varepsilon_1}\cdots y_n^{\varepsilon_n}}=xe^{y_1^{\varepsilon_1}\cdots y_n^{\varepsilon_n}}=xy.
\end{eqnarray*}
Hence, it holds that 
\begin{eqnarray*}
f_{\#}(xy)&=&f_{\#}(x^{y_1^{\varepsilon_1}\cdots y_n^{\varepsilon_n}})\\
&=&f_{\#}(x)^{f_{\#}(y_1)^{\varepsilon_1}\cdots f_{\#}(y_n)^{\varepsilon_n}}\\
&=&f_{\#}(x)e^{f_{\#}(y_1)^{\varepsilon_1}\cdots f_{\#}(y_n)^{\varepsilon_n}}\\
&=&f_{\#}(x)f_{\#}(e)^{f_{\#}(y_1)^{\varepsilon_1}\cdots f_{\#}(y_n)^{\varepsilon_n}}\\
&=&f_{\#}(x)f_{\#}(e^{y_1^{\varepsilon_1}\cdots y_n^{\varepsilon_n}})=f_{\#}(x)f_{\#}(y).
\end{eqnarray*}
 This implies that $f_{\#}$ is a group automorphism.
\end{proof}
Let $f:{\rm GAlex}(G,\varphi)\to{\rm GAlex}(G,\varphi)$ be a quandle automorphism and ${\rm Aut}(G)$ the automorphism group of $G$. By the proof of Lemma~\ref{lemm:quandle_auto}, the group automorphism $f_{\#}:G\to G;x\mapsto f(x)f(e)^{-1}$ satisfies that $\varphi\circ f_{\#}=f_\#\circ\varphi$. Thus, we define the subgroup $C_{{\rm Aut}(G)}(\varphi)$ by $C_{{\rm Aut}(G)}(\varphi):=\{ \psi\in{\rm Aut}(G)\mid \psi\circ\varphi=\varphi\circ\psi\}$. It is easily seen that the subset $\{ r\in G\mid \varphi(r)=r,r^2=e\}$ is closed under the action of $C_{{\rm Aut}(G)}(\varphi)$. 
\begin{theo}
\label{main_theo:2}
Let $G$ be a group and $\varphi:G\to G$ a group automorphism. Suppose that ${\rm GAlex}(G,\varphi)$ is a kei and a connected quandle. Then, there exists a bijection 
\[
\mathcal{SQ}({\rm GAlex}(G,\varphi))\overset{1:1}{\leftrightarrow}\{r\in G\mid \varphi(r)=r,r^2=e\}/C_{{\rm Aut}(G)}(\varphi).
\]
\end{theo}
\begin{proof}
Let $r_1$ and $r_2$ be elements of $\{r\in G\mid \varphi(r)=r,r^2=e\}$. Assume that there is an element $f\in C_{{\rm Aut}(G)}(\varphi)$ such that $f(r_1)=r_2$. Notice that $f$ is a quandle automorphism of ${\rm GAlex}(G,\varphi)$. Let $\rho_{r_1}$ and $\rho_{r_2}$ be good involutions of ${\rm GAlex}(G,\varphi)$ defined by $\rho_{r_i}(x)=r_ix~(i=1,2)$. It holds that 
\[
f(\rho_{r_1}(x))=f(r_1x)=f(r_1)f(x)=r_2f(x)=\rho_{r_2}(f(x))
\]
for any $x\in {\rm GAlex}(G,\varphi)$. Thus, the symmetric quandles $( {\rm GAlex}(G,\varphi),\rho_{r_1})$ and $( {\rm GAlex}(G,\varphi),\rho_{r_2})$ are symmetric quandle isomorphic.

Let $\rho_1$ and $\rho_2$ be good involutions of ${\rm GAlex}(G,\varphi)$. Suppose that there is a symmetric quandle isomorphism from $({\rm GAlex}(G,\varphi),\rho_{r_1})$ to $( {\rm GAlex}(G,\varphi),\rho_{r_2})$, which is denoted by $f$. Since $f$ is a quandle automorphism of ${\rm GAlex}(G,\varphi)$, the map $f_{\#}:G\to G;x\mapsto f(x)f(e)^{-1}$ is an element of $C_{{\rm Aut}_G}(\varphi)$ (Lemma~\ref{lemm:quandle_auto}). Putting $r_1:=\rho_1(e)$ and $r_2:=\rho_2(e)$, we have $\rho_1(x)=r_1x$ and $\rho_2(x)=r_2x$. Then, for any $x\in {\rm GAlex}(G,\varphi)$, it holds that 
\[
f(\rho_1(x))=f_\#(r_1x)f(e)^{-1}=f_\#(r_1)f_\#(x)f(e)^{-1}
\] and 
\[
\rho_2(f(x))=r_2f_\#(x)f(e)^{-1}.\]
 Since $f\circ\rho_1=\rho_2\circ f$, we have $f_\#(r_1)=r_2$. Thus, we see that the bijection constructed in Proposition~\ref{prop:good_inv} induces a biijection between $\mathcal{SQ}({\rm GAlex}(G,\varphi))$ and $\{r\in G\mid \varphi(r)=r,r^2=e\}/C_{{\rm Aut}(G)}(\varphi)$.
\end{proof}
\begin{exam}
Let $n$ be a positive integer. We regard $S^1$ as $\R/\Z$, and denote $(S^1)^n$ by $T_n$. Let us consider good involutions of the Alexander quandle $X={\rm GAlex}(T_n,{\rm inv}(T_n))$. We note that the quandle $X$ is a kei and a connected quandle. Recall that ${\rm inv}(T_n)(x)=-x$ for any $x\in T_n$. Hence, we see that
\begin{eqnarray*}
&&\{r\in T_n\mid {\rm inv}(T_n)(r)=r,2r=0\}\\
&&=\{(x_1,\ldots,x_n)\in T_n\mid x_i\in\{0,1/2\}\ (i=1,\ldots,n) \}.
\end{eqnarray*}
By Proposition~\ref{prop:good_inv}, the quandle $X$ has $2^n$ good involutions.

Next, let us consider the set $\mathcal{SQ}(X)$. We define the map $E_{ij}:T_n\to T_n$ by
\[
E_{ij}(x_1,\ldots,x_i,\ldots,x_j,\ldots,x_n)=(x_1,\ldots,x_i+x_j,\ldots,x_j,\ldots,x_n).
\] 
Then, we see that $E_{ij}$ is an element of $C_{{\rm Aut}(T_n)}({\rm inv}(T_n))$. Let $H$ be the subgroup of $C_{{\rm Aut}(T_n)}({\rm inv}(T_n))$ generated by $\{E_{ij}\mid i\neq j\}$. It is not difficult to check that the following equality holds:
\begin{eqnarray*}
&&\{\psi(1/2,0,\ldots,0)\in T_n\mid \psi\in H\}\\
&&=\{(x_1,\ldots,x_n)\in T_n\mid x_i\in\{0,1/2\}\ (i=1,\ldots,n) \}\backslash\{(0,\ldots,0)\}.
\end{eqnarray*}
Since an element of $C_{{\rm Aut}(T_n)}({\rm inv}(T_n))$ is a group automorphism, the orbit of $(0,\ldots,0)$ is equal to $\{(0,\ldots,0)\}$. 

Hence, the set $\{r\in T_n\mid {\rm inv}(T_n)(r)=r,2r=0\}/C_{{\rm Aut}(T_n)}({\rm inv}(T_n))$ consists of 2 elements. By Theorem~\ref{main_theo:2}, the cardinality of $\mathcal{SQ}(X)$ is equal to $2$.
\end{exam}
\section*{Acknowledgements}
The author would like to thank Seiichi Kamada for his several comments. This work was supported by JSPS KAKENHI Grant Number 21J21482.

\end{document}